\documentclass[leqno,twoside]{amsart}

\usepackage{latexsym,esint}

\theoremstyle{plain}

\def\endproof{\hspace*{\fill}\mbox{\ \rule{.1in}{.1in}}\medskip }

\newtheorem{theorem}{Theorem}[section]

\newtheorem{lemma}[theorem]{Lemma}
\newtheorem{proposition}[theorem]{Proposition}

\theoremstyle{definition}
\newtheorem{example}[theorem]{Example}
\newtheorem{remark}[theorem]{Remark}

\numberwithin{equation}{section}
\numberwithin{figure}{section}

\begin{document}

\title[optimal constants in rigidity estimates]
{On the optimal constants in Korn's and geometric rigidity
  estimates, \\in bounded and unbounded domains,\\
under Neumann boundary conditions}
\author{Marta Lewicka and Stefan M\"uller}
\address{University of Pittsburgh, Department of Mathematics, 
301 Thackeray Hall, Pittsburgh, PA 15260, USA}
\email{lewicka@pitt.edu}
\address{Institut f\"ur Angewandte Mathematik, Universit\"at Bonn,
  Endenicher Allee 60,  53115 Bonn, Germany}
\email{stefan.mueller@hcm.uni-bonn.de}

\maketitle


\begin{abstract} 
We are concerned with the optimal constants:  in the Korn inequality
under tangential boundary conditions on bounded sets $\Omega \subset
\mathbb{R}^n$, and in the geometric
rigidity estimate on the whole $\mathbb{R}^2$. We prove that the latter
constant equals $\sqrt{2}$, and we discuss the relation of the former
constants with the optimal Korn's constants under Dirichlet boundary
conditions, and in the whole $\mathbb{R}^n$, which are well known to equal
$\sqrt{2}$. We also discuss the attainability of these constants and
the structure of deformations/displacement fields in the optimal sets.
\end{abstract}

\section{Introduction and the main results} 

In this paper we are concerned with the optimal constants in the Korn
inequality \cite{korn1, korn2} and in the Friesecke-James-M\"uller
geometric rigidity estimate \cite{FJMgeo, FJMhier}.

Let $\Omega$ be an open, bounded, and connected subset
of $\mathbb{R}^n$ with Lipschitz continuous boundary. 
The Korn inequality \cite{korn1, korn2, OK} states
that there exists a constant $C(\Omega)$ depending only on
$\Omega$, such that for all $u\in W^{1,2}(\Omega,\mathbb{R}^n)$ there
holds:
\begin{equation}\label{korn_general}
\min_{A\in so(n)} \|\nabla u - A\|_{L^2(\Omega)} \leq C(\Omega)
\|D(u)\|_{L^2(\Omega)},
\end{equation}
where by $D(u)= \frac{1}{2} \left(\nabla u + (\nabla u)^T\right)$ 
we mean the symmetric part of $\nabla u$.

Let now $\vec n$ denote the outward unit normal on $\partial\Omega$.  
Given (\ref{korn_general}), it is not hard to deduce (see Lemma
\ref{projection}) the following
version of Korn's inequality subject to tangential boundary
conditions. Namely, there exists a constant $\kappa(\Omega)$,
depending only on $\Omega$, such that for all
$u\in W^{1,2}(\Omega,\mathbb{R}^n)$ satisfying $u\cdot \vec n =0$ 
on $\partial\Omega$ there holds:
\begin{equation}\label{korn}
\min_{A\in L_\Omega} \|\nabla u - A\|_{L^2(\Omega)} \leq \kappa(\Omega)
\|D(u)\|_{L^2(\Omega)},
\end{equation}
where by $L_\Omega$  above we denote the linear space of 
skew-symmetric matrices that are gradients of affine maps tangential 
on the boundary of $\Omega$:
$$L_\Omega = \left\{A\in so(n); ~~ \exists a\in \mathbb{R}^n \quad 
\forall x\in\partial\Omega \quad (Ax + a)\cdot\vec n(x) = 0\right\}.$$
The optimal constant in (\ref{korn}) is given by:
\begin{equation}\label{def_kappa}
\begin{split}
\kappa(\Omega) = \sup\Bigg\{\min_{A\in L_\Omega}\|\nabla u - A\|_{L^2(\Omega)};
~~ & u\in W^{1,2}(\Omega,\mathbb{R}^n), ~ u\cdot\vec n = 0 
\mbox{ on } \partial\Omega\\ 
& \mbox{ and } \|D(u)\|_{L^2(\Omega)}=1  \Bigg\},
\end{split}
\end{equation}
and we aim to study its relation to Korn's constant in the
whole $\mathbb{R}^n$, which is $\sqrt{2}$ (see Lemma \ref{lem222}):
\begin{equation}\label{def_n}
\kappa(\mathbb{R}^n) = \sup\Bigg\{\|\nabla u\|_{L^2(\mathbb{R}^n)};
~~  u\in W^{1,2}(\mathbb{R}^n,\mathbb{R}^n), ~
\|D(u)\|_{L^2(\mathbb{R}^n)}=1  \Bigg\} = \sqrt{2}.
\end{equation}
In this setting, our first set of main results is:
\begin{theorem}\label{th_uno}
For any open, bounded, Lipschitz, connected $\Omega\subset\mathbb{R}^n$:
\begin{equation}\label{ineq}
\kappa(\Omega)\geq \kappa(\mathbb{R}^n) =\sqrt{2}.
\end{equation}
\end{theorem}
In fact, $\kappa(\Omega)$ may be arbitrarily large. 
In Example \ref{blowupi} we will recall our construction in \cite{LM} which
implies that for a sequence of thin shells around a sphere, the
Korn constants go to $\infty$  as the thickness goes to $0$. On the
other hand, as we show in Example \ref{square}, there is: $\kappa([0,1]^2) = \sqrt{2}$. \footnote{We are able
  to prove that for smooth domains there always holds: $\kappa(\Omega)
  > \sqrt{2}$. The proof of this fact will appear elsewhere. }
We however have:
\begin{theorem}\label{th_due}
Assume that there exists a sequence $\{u_k\}_{k=1}^\infty$, 
$u_k\in W^{1,2}(\Omega,\mathbb{R}^n)$, $u_k\cdot\vec n = 0$ on $\partial\Omega$,
with the following properties:
\begin{itemize}
\item[(i)] $ u_k$ converges to $0$ weakly in $W^{1,2}(\Omega,\mathbb{R}^n)$,
\item[(ii)] $\|D(u_k)\|_{L^2(\Omega)} = 1$,
\item[(iii)] $\lim_{k\to\infty}\|\nabla u_k\|_{L^2(\Omega)} = \kappa(\Omega)$.
\end{itemize}
Then $\kappa(\Omega) = \sqrt{2}$.
\end{theorem}

\begin{theorem}\label{cor1}
If $\kappa(\Omega) > \sqrt{2}$ then the supremum in the definition
(\ref{def_kappa}) is attained.
More precisely, for every $A_0\in L_\Omega$ there exists 
$u\in W^{1,2}(\Omega, \mathbb{R}^n)$, such that
$u\cdot\vec n = 0$ on $\partial\Omega$, $D(u) \not\equiv 0$, and:
\begin{equation}\label{exact}
\min_{A\in so(n)}\|\nabla u - A\|_{L^2(\Omega)} = 
\|\nabla u - A_0\|_{L^2(\Omega)}
= \kappa(\Omega) \|D(u)\|_{L^2(\Omega)}.
\end{equation}
\end{theorem}

\begin{theorem}\label{cor2}
The vector fields $u$ for which Korn's constant $\kappa(\Omega)$ is attained:
\begin{equation}\label{space_exact}
\Big\{u\in W^{1,2}(\Omega,\mathbb{R}^n); ~~ u\cdot\vec n = 0 \mbox{ on }
\partial\Omega, ~ u \mbox{ satisfies } (\ref{exact}) \mbox{ for some }
A_0\in L_\Omega \Big\};
\end{equation}
form a closed linear subspace of $W^{1,2}(\Omega,\mathbb{R}^n)$.
Moreover, if $\kappa(\Omega)>\sqrt{2}$ then this space is of finite
dimension.
\end{theorem}

\medskip

In the second part of this paper we concentrate on the nonlinear
version of Korn's inequality, namely the Friesecke-James-Muller
geometric rigidity estimate \cite{FJMgeo, FJMhier}. It states that for an open, bounded, smooth and
connected domain $\Omega\subset\mathbb{R}^n$, there exists a constant
$\kappa_{nl}(\Omega)$ depending only on $\Omega$, such that for every $u\in
W^{1,2}(\Omega,\mathbb{R}^n)$ there holds:
\begin{equation}\label{FJM-intro}
\min_{R\in SO(n)} \|\nabla u - R\|_{L^2(\Omega)}
\leq \kappa_{nl}(\Omega) \|\mbox{dist}(\nabla u, SO(n))\|_{L^2(\Omega)}.
\end{equation}
Define:
\begin{equation}\label{defik}
\begin{split}
 \kappa_{nl}(\mathbb{R}^n) = \sup\Bigg\{&\min_{R\in SO(n)}
  \frac{\|\nabla u - R\|_{L^2(\mathbb{R}^n)}}{\|\mbox{dist}(\nabla u, SO(n))\|_{L^2(\mathbb{R}^n)}};
~~ u\in W^{1,2}_{loc}(\mathbb{R}^n, \mathbb{R}^n), \\ 
& \qquad \qquad \qquad \qquad\quad
\mathrm{dist}(\nabla u, SO(n))\in L^2(\mathbb{R}^n) \setminus \{0\} \Bigg\}.
\end{split}
\end{equation}
Our results in this context are restricted to dimension $2$:
\begin{theorem}\label{maxmin}
We have: $\kappa_{nl}(\mathbb{R}^2) = \sqrt{2}$.
In particular: 
\begin{equation*}
\begin{split}
\forall u\in W^{1,2}_{loc}(\mathbb{R}^2,\mathbb{R}^2) &\quad 
\mathrm{dist}(\nabla u, SO(2)) \in L^2(\mathbb{R}^2) \implies\\
& \min_{R\in SO(n)} \|\nabla u - R\|_{L^2(\mathbb{R}^2)}
\leq \sqrt{2} \|\mathrm{dist}(\nabla u, SO(2))\|_{L^2(\mathbb{R}^2)}.
\end{split}
\end{equation*}
\end{theorem}

\begin{theorem}\label{th2}
For every rotation $R_0\in SO(2)$ there exists $u\in
W^{1,2}_{loc}(\mathbb{R}^2, \mathbb{R}^2)$ with $\mathrm{dist}(\nabla
u, SO(2))\in L^2(\mathbb{R}^2) \setminus \{0\}$ such that:
\begin{equation}\label{attain}
\begin{split}
\min_{R\in SO(2)} \|\nabla u - R\|_{L^2(\mathbb{R}^2)} & =
\|\nabla u - R_0\|_{L^2(\mathbb{R}^2)} \\ & 
= \sqrt{2} \| \mathrm{dist}(\nabla u(x), SO(2))\|_{L^2(\mathbb{R}^2)}.
\end{split}
\end{equation}
\end{theorem}

\begin{theorem}\label{th3}
The vector fields for which the nonlinear Korn constant in
(\ref{defik}) is attained, namely:
\begin{equation*} 
\begin{split}
\Big\{u\in W^{1,2}_{loc}(\mathbb{R}^2, \mathbb{R}^2); &~~ \mathrm{dist}(\nabla
u, SO(2))\in L^2(\mathbb{R}^2), \\ & ~ u \mbox{ satisfies } (\ref{attain}) \mbox{ for
  some } R_0 \in SO(2) \Big\}
\end{split}
\end{equation*} 
have the defining property that their gradients are of the form:
\begin{equation}\label{dodici}
\nabla u(x) = R_0R(\alpha(x)) + \left[\begin{array}{cc} a(x) & b(x)\\
b(x) & -a(x)\end{array}\right] \quad \mbox{ with } R(\alpha) = 
\left[\begin{array}{cc} \cos \alpha & -\sin\alpha\\
\sin\alpha & \cos\alpha\end{array}\right],
\end{equation}
for some $\alpha, a, b\in L^2(\mathbb{R}^2)$. Conversely, for every
$\alpha\in L^2(\mathbb{R}^2)$ there exists $a,b\in  L^2(\mathbb{R}^2)$ and $u\in
W^{1,2}_{loc}(\mathbb{R}^2, \mathbb{R}^2)$ such that
(\ref{attain}) and (\ref{dodici}) hold.
\end{theorem}

The proofs of the three Theorems above are independent from the proof of
(\ref{FJM-intro}) in \cite{FJMgeo}. They rely on the
conformal-anticonformal decomposition of $2\times 2$ matrices, and it
is not clear how this construction and methods could be extended to yield a result
in higher dimensions $n>2$. 

There is an extensive literature relating to Korn's inequality and its
applications, notably in linear elasticity \cite{JC, ciarbookvol1, korn1, OK}. 
On the other hand, the
nonlinear estimate (\ref{FJM-intro}) plays crucial role in models in
nonlinear elasticity \cite{FJMhier, FJMgeo}. Indeed, the relation between
these two estimates is clear if we recall that 
the tangent space to $SO(n)$ at $Id$ is $so(n)$.
The blow-up rate and properties of $\kappa(\Omega)$ for thin
spherical-like domains around a given surface were studied 
in \cite{LM}. The relations of $\kappa(\Omega)$ with the measure of
axisymmetry of $\Omega$ have been discussed in \cite{DV}.
An interesting extension  of both Korn's and the geometric
rigidity estimates under mixed growth
conditions has been recently established in \cite{ConDolzMul}.

\section{Preliminaries}

Recall that the linear space of skew-symmetric matrices is:
$$so(n) = \{A\in \mathbb{R}^{n\times n}; ~ A = -A^T\}$$
while $SO(n)$ stands for the group of proper rotations:
$$SO(n) = \{R\in \mathbb{R}^{n\times n}; R^T = R^{-1} \mbox{ and } \det
R = 1\}.$$
The scalar product and the (Frobenius) norm in the space of $n\times n$ (real)
matrices $\mathbb{R}^{n\times n}$ are given by:
$$A:B = \mbox{tr}(A^T B)\qquad |A|^2 = A:A.$$

We first notice the following characterization of the minimiser 
in (\ref{korn}):
\begin{lemma}\label{projection}
Let $u\in W^{1,2}(\Omega,\mathbb{R}^n)$, $u\cdot\vec n=0$ on $\partial\Omega$.
Then the minimum in the left hand side of (\ref{korn}) is attained, uniquely, at:
$$A_0 = \mathbb{P}_{L_\Omega}\fint_{\Omega}\nabla u,$$
where $\mathbb{P}_{L_\Omega}$ denotes the orthogonal projection of
$\mathbb{R}^{n\times n}$  on $L_\Omega$.
\end{lemma}
\begin{proof}
Let $A_0\in L_\Omega$ be a minimiser of $\|\nabla u - A\|^2_{L^2(\Omega)}$
over $L_\Omega$.  Taking the  derivative 
in the direction of $A\in L_\Omega$, one obtains:
$$\forall A\in L_\Omega \qquad \int_\Omega (\nabla u - A_0): A = 0.$$
Equivalently, there holds:
$$\left(\fint_\Omega \nabla u - A_0\right)\in L_\Omega^\perp,$$
which implies the lemma.
\end{proof}

For convenience of the reader, we now sketch the proof of (\ref{def_n}).

\begin{lemma}\label{lem222}
For every open, Lipschitz,  connected $\Omega \subset
\mathbb{R}^n$, the Korn constant under Dirichlet boundary conditions
equals $\sqrt{2}$:
\begin{equation}\label{korn-dir}
 \kappa_{0}(\Omega) = \sup\left\{\|\nabla u\|_{L^2(\Omega)};
~~ u\in W^{1,2}_0(\Omega, \mathbb{R}^n), 
~ \|D(u)\|_{L^2(\Omega)} = 1\right\} = \sqrt{2}.
\end{equation}
\end{lemma}
\begin{proof}
For every $u\in W^{1,2}(\Omega, \mathbb{R}^n)$ we have:
\begin{equation}\label{strange}
\begin{split}
2\int_\Omega|D(u)|^2 & = \int|\nabla u|^2 + \int\nabla u : (\nabla u)^T = \int
|\nabla u|^2 + \int \mbox{tr}(\nabla u)^2 \\ & 
= \int|\nabla u|^2 + \int
|\mbox{div } u|^2 + \int \left(\mbox{tr}(\nabla u)^2 - (\mbox{tr}
  \nabla u)^2\right).
\end{split}
\end{equation}
When, additionally, $\Omega$ is bounded and $u\in W^{1,2}_0(\Omega,\mathbb{R}^n)$, this
implies that:
\begin{equation*}
\begin{split}
2\int_\Omega|D(u)|^2 = \int|\nabla u|^2 + \int |\mbox{div } u|^2,
\end{split}
\end{equation*}
because $\left(\mbox{tr}(\nabla u)^2 - (\mbox{tr}
  \nabla u)^2\right) $ is a null-Lagrangean, i.e. its integral depends
only on the boundary value of $u$ on $\partial \Omega$. 
We therefore conclude that, in this case:  $\|\nabla u\|_{L^2(\Omega)} \leq
\sqrt{2}  \|D(u)\|_{L^2(\Omega)}$. The same inequality is also true on
unbounded domains, because of the density of
$\mathcal{C}^\infty_c(\Omega,\mathbb{R}^n)$ in $W_0^{1,2}(\Omega, \mathbb{R}^n)$.

To prove that $\sqrt{2}$ is optimal and
that it is attained, it is enough to take
$u\in\mathcal{C}_c^\infty(\Omega, \mathbb{R}^n)$  with $\mbox{div } u =
0$ (when $n=3$, take $u=\mbox{curl } v$ for any compactly supported
$v$). 
This achieves the proof.
\end{proof}

\medskip

We now recall the Poincar\'e inequality for tangential vector fields.
The proof, which can be found in \cite{bishop}, is deduced through 
a standard argument by contradiction.
\begin{lemma}\label{poincare}
Let $\Omega\subset\mathbb{R}^n$ be an open, bounded, Lipschitz set. 
For every $u\in W^{1,2}(\Omega,\mathbb{R}^n)$, $u\cdot \vec n = 0$ 
on $\partial\Omega$, there holds:
\begin{equation}\label{poinc}
\|u\|_{L^2(\Omega)} \leq C(\Omega) \|\nabla u\|_{L^2(\Omega)},
\end{equation}
where the constant $C(\Omega)$ depends only on $\Omega$  (it is independent of $u$).
\end{lemma}

\section{The optimal Korn constant $\kappa(\Omega)$: a proof of
  Theorem \ref{th_uno} and two examples}

In the course of proof of Theorem \ref{th_uno}, we will 
use the following observation:
\begin{proposition}\label{tech}
For any $f\in L^2(\mathbb{R}^n)$ there holds:
$$\lim_{R\to\infty} R^{-n/2} \|f\|_{L^1(B_R)} = 0,$$
on the ball  $B_R = \{x\in\mathbb{R}^n, ~ |x|\leq R\}$.
\end{proposition}
\begin{proof}
Fix $\epsilon>0$.  For $m$ sufficiently large, one has
$\|f\|_{L^2(\mathbb{R}^n\setminus B_m)} <\epsilon$. 
Denote by $\omega_n$ the volume of the unit ball $B_1$ in $\mathbb{R}^n$.
Take any $R>m$ so that: $\displaystyle 
\left(\frac{m}{R}\right)^{n/2}\|f\|_{L^2(\mathbb{R}^n)} \leq \epsilon.$
Then:
\begin{equation*}
\begin{split}
R^{-n/2} \|f\|_{L^1(B_R)} & = R^{-n/2} \left(\int_{B_R\setminus B_m} |f|
+ \int_{B_m} |f|\right) \\
& \leq R^{-n/2} |B_R|^{1/2}\epsilon + 
R^{-n/2} |B_m|^{1/2} \|f\|_{L^2(\mathbb{R}^n)}\\
& \leq \omega_n^{1/2}\epsilon + \left(\frac{m}{R}\right)^{n/2}\omega_n^{1/2}
 \|f\|_{L^2(\mathbb{R}^n)} \leq 2\omega_n^{1/2} \epsilon,
\end{split}
\end{equation*}
which achieves the proof.
\end{proof}

\medskip

\noindent{\bf Proof of Theorem \ref{th_uno}}

\noindent {\bf 1.} 
Without loss of generality we may assume that $0\in\Omega$.
Let $u\in W^{1,2}(\mathbb{R}^n, \mathbb{R}^n)$ with 
$\|D(u)\|_{L^2(\mathbb{R}^n)} = 1$. Define the sequence 
$u_k\in W^{1,2}(\mathbb{R}^n, \mathbb{R}^n)$ by:
$u_k(x) = k^{n/2-1} u(kx)$. One has:
$$\|\nabla u_k\|_{L^2(\mathbb{R}^n)} = \|\nabla u\|_{L^2(\mathbb{R}^n)},
\qquad \|D(u_k)\|_{L^2(\mathbb{R}^n)} = \|D(u)\|_{L^2(\mathbb{R}^n)} = 1.$$
Let now $\phi\in\mathcal{C}_c^\infty(\Omega)$ be a nonnegative function, 
equal identically to $1$ in a neighborhood of $0$,
and define: $v_k=\phi u_k$. Clearly $v_k\in W^{1,2}_0(\Omega,\mathbb{R}^n)$ and:
$$\nabla v_k = \phi\nabla u_k + u_k \otimes \nabla \phi.$$
We claim that:
\begin{eqnarray}
&&  \lim_{k\to\infty} \|\nabla v_k\|_{L^1(\Omega)} = 0, \label{uno}\\
&&  \lim_{k\to\infty} \|\nabla v_k\|_{L^2(\Omega)} 
= \|\nabla u\|_{L^2(\mathbb{R}^n)}, \label{due}\\
&&  \lim_{k\to\infty} \|D(v_k)\|_{L^2(\Omega)} 
= \|D(u)\|_{L^2(\mathbb{R}^n)} = 1. \label{tre}
\end{eqnarray}
To prove the claim, 
notice first that:
$$\lim_{k\to\infty} \|u_k\otimes\nabla\phi\|_{L^2(\Omega)}
\leq \lim_{k\to\infty} 
\|\nabla\phi\|_{L^\infty} k^{-1} \|u\|_{L^2(\mathbb{R}^n)} =0.$$
On the other hand, for all $i,j:1\ldots n$:
$$\lim_{k\to\infty} \left\|\phi\frac{\partial}{\partial x_i} u_k^j\right\|^2_{L^2(\Omega)}
= \lim_{k\to\infty} \int_{\mathbb{R}^n} \left|\phi (x/k)
 \frac{\partial}{\partial x_i} u^j(x)\right|^2~\mbox{d}x
= \left\|\frac{\partial}{\partial x_i} u^j\right\|^2_{L^2(\Omega)}.$$
Thus we obtain (\ref{due}) and (\ref{tre}). Similarly:
$$\lim_{k\to\infty} \left\|\phi\nabla u_k\right\|_{L^1(\Omega)}
\leq \lim_{k\to\infty} \|\phi\|_{L^\infty} k^{-n/2} \|\nabla u\|_{L^1(k\Omega)}
=0,$$ 
where the last equality follows by Proposition \ref{tech}. Hence we conclude
(\ref{uno}) as well.

\medskip

{\bf 2.} Notice that by Lemma \ref{projection}:
\begin{equation*}
\begin{split}
& \min_{A\in L_\Omega} \|\nabla v_k - A\|_{L^2(\Omega)} = 
\left\|\nabla v_k - \mathbb{P}_{L_\Omega} \fint_{\Omega} \nabla v_k\right\|_{L^2(\Omega)}\\
& \qquad 
\geq \|\nabla v_k\|_{L^2(\Omega)} - 
\left\|\mathbb{P}_{L_\Omega} \fint_{\Omega} \nabla v_k\right\|_{L^2(\Omega)}
\geq \|\nabla v_k\|_{L^2(\Omega)} - |\Omega|^{-1/2} \|\nabla v_k\|_{L^1(\Omega)}.
\end{split}
\end{equation*}
Now, by (\ref{uno}) and (\ref{due}), the right hand side of the above 
inequality converges to $\|\nabla u\|_{L^2(\mathbb{R}^n)}$ as $k\to \infty$.
On the other hand, by (\ref{korn}) and (\ref{def_kappa}), the left hand side 
is bounded by $\kappa(\Omega) \|D(v_k)\|_{L^2(\Omega)}$. Therefore,
passing to the limit and using (\ref{tre}), we
obtain:
$$\|\nabla u\|_{L^2(\mathbb{R}^n)}\leq \kappa(\Omega)
\|D(u)\|_{L^2(\mathbb{R}^n)} = \kappa(\Omega).$$
Recalling the definition (\ref{def_n}) the theorem follows.
\endproof

\bigskip

\begin{example}\label{square}
We now show that $\kappa(Q) = \sqrt{2}$ for $Q = [0,1]^2\subset\mathbb{R}^2$.

Firstly, observe that (see Theorem 9.4 \cite{LM})
$L_\Omega\neq\{0\}$ if and only if $\Omega$ has a rotational
symmetry. When this is not the case, then:
\begin{equation}\label{norot}
\kappa(\Omega) = \sup\Bigg\{\frac{\|\nabla u\|_{L^2(\Omega)}}{\|D(u)\|_{L^2(\Omega)}};
~~~ u\in W^{1,2}(\Omega,\mathbb{R}^n), ~ u\cdot\vec n = 0 
\mbox{ on } \partial\Omega\Bigg\}.
\end{equation}
In view of (\ref{norot}) and Theorem \ref{th_uno}, it is hence enough to
prove that for every $u\in
W^{1,2}(Q,\mathbb{R}^2)$ satisfying $u^1(0, x_2) = u^1(1, x_2) = 0$
and $ u^2(x_2, 0) = u^2(x_1, 0) = 0$ for all $x_1, x_2\in [0,1]$,
there holds:
\begin{equation}\label{ini}
\int_Q |\nabla u|^2 \leq 2\int_Q |D(u)|^2.
\end{equation}

Consider first a regular vector field $u\in\mathcal{C}^2(\bar Q,\mathbb{R}^2). $
As in (\ref{strange}), we obtain:
\begin{equation}\label{strange2}
\int_Q |D(u)|^2 = \frac{1}{2}\int_Q |\nabla u|^2 + \frac{1}{2}\int_Q
|\mbox{div } u|^2 + \int_Q \big(\partial_1u^2\partial_2u^1 - \partial_1u^1\partial_2u^2\big).
\end{equation}
Note that:
\begin{equation*} 
\begin{split}
\int_Q \big(\partial_1 u^2\partial_2 u^1 - \partial_1u^1\partial_2u^2
\big) 
= \int_Q \partial_1(u^2\partial_2 u^1) -
\int_Q \partial_2(u^2\partial_{1}u^1),
\end{split}
\end{equation*}
and that both terms in the right hand side of the
above equality integrate to $0$ on $Q$, because of the assumed boundary
condition. Thus, (\ref{strange2}) yields (\ref{ini}) for $u\in\mathcal{C}^2$.

\smallskip

It now suffices to check that every $u\in W^{1,2}(Q, \mathbb{R}^2)$ with $u\cdot
\vec n = 0$ on $\partial Q$, can be approximated by a sequence of
$\mathcal{C}^2(\bar Q, \mathbb{R}^2)$ vector fields satisfying the same boundary condition.
To this end, define the extension $\bar u^1\in W^{1,2}([0,1]\times [-1, 2],
\mathbb{R})$ of the component $u^1\in W^{1,2}(Q,\mathbb{R})$,  by:
$$ \forall x_1\in[0,1] \quad \forall x_2\in[-1,2]\qquad \bar u^1(x) =
\left\{\begin{array}{ll} u^1(x) & \mbox{if } ~x_2\in [0,1]\\
u^1(x_1,-x_2) & \mbox{if } ~x_2\in [-1,0]\\
u^1(x_1,2-x_2) & \mbox{if } ~x_2\in [1,2].
\end{array}\right.$$
Let $\phi:(-1, 2)\to\mathbb{R}$ be a nonnegative, smooth and compactly supported
function, equal to $1$ on $[0, 1]$. Then $\phi\bar u^1\in W^{1,2}_0([0,1]\times [-1,
2], \mathbb{R})$, and thus $\phi\bar u^1$ can be approximated in $W^{1,2}$
by a sequence $u_k^1\in\mathcal{C}_c^\infty([0,1]\times [-1, 2],
\mathbb{R})$. Clearly, $u_k^1$ converges to $u^1$ on $Q$, and each
$u_k^1(x) = 0$ whenever $x_1\in\{0, 1\}$. 

In a similar manner, we construct smooth approximating sequence
$\{u_k^2\}_{k\geq 1}$. Writing $u_k = (u_k^1, u_k^2)\in \mathcal{C}^\infty(\bar
Q,\mathbb{R}^2)$, we obtain the desired approximations of $u$.
\endproof
\end{example}

\medskip

\begin{example}\label{blowupi}
We now recall the construction \cite{LM} of a
family of domains $\Omega^h\subset\mathbb{R}^n$
parametrised by  $0<h\ll 1$, with the property that:
$$\kappa(\Omega^h)\to \infty\quad\mbox{ as }h\to 0.$$ 
Let $S$ denote the $(n-1)$-dimensional unit sphere in $\mathbb{R}^n$ and
let $g:S\to (0, \frac{1}{3})$ be a smooth function on $S$. Define:
$$\Omega^h = \Big\{ (1+t)x; ~~x\in S, ~ t\in\big( hg(x)-h,
hg(x)\big)\Big\}.$$
Clearly, we may request from function $g$ to be such that no $\Omega^h$ has
any rotational symmetry, and hence $L_{\Omega^h}=\{0\}$ implies
(\ref{norot}) for all $h$.

Let now $v:S\to \mathbb{R}^n$ be a tangent vector field given by a rotation:
$v(x) = a\times x$, for some $a\in\mathbb{R}^n$. Define $u^h\in
W^{1,2}(\Omega^h, \mathbb{R}^n)$:
$$u^h(x+tx) =\Big( (1+t) \mbox{Id}  + hx\otimes \nabla g(x)\Big)v(x) =
(1+t) (a\times x) + \langle a, x\times \nabla g(x)\rangle x. $$
One can check that $u^h$ is tangent at $\partial\Omega^h$ and that:
$$\|\nabla u^h\|_{L^2(\Omega^h)}\geq C h^{1/2}, \qquad 
\|D(u^h)\|_{L^2(\Omega^h)}\leq C h^{3/2}.$$
Hence we conclude the blow-up of Korn's constant:
$\kappa(\Omega^h) \geq C{h}^{-1}$ in the vanishing thickness $h\to 0$. \endproof
\end{example}

\section{The optimal Korn constant $\kappa(\Omega)$: proofs of
  Theorems \ref{th_due}, \ref{cor1} and \ref{cor2}}

\noindent{\bf Proof of Theorem \ref{th_due}}

\noindent {\bf 1.} From (ii) and (iii) we see that the sequences:
$$\{|\nabla u_k|^2\chi_\Omega ~\mbox{d}x\}_{k=1}^\infty
\quad \mbox{and} \quad \{|D(u_k)|^2\chi_\Omega ~\mbox{d}x\}_{k=1}^\infty$$ 
are bounded in the space of Radon measures $\mathcal{M}(\mathbb{R}^n)$. 
Therefore (possibly passing to subsequences), they converge weakly 
in $\mathcal{M}(\mathbb{R}^n)$ to some $\mu,\nu$, concentrated on $\bar\Omega$.
That is:
\begin{equation}\label{p_uno}
\begin{split}
\forall \phi\in\mathcal{C}_c^\infty (\mathbb{R}^n) \qquad
 \lim_{k\to\infty} \int_\Omega \phi^2 |\nabla u_k|^2 ~\mbox{d}x & =
  \int_{\mathbb{R}^n} \phi^2 ~\mbox{d}\mu,\\ 
 \lim_{k\to\infty} \int_\Omega \phi^2 |D(u_k)|^2 ~\mbox{d}x & =
  \int_{\mathbb{R}^n} \phi^2 ~\mbox{d}\nu.
\end{split}
\end{equation}
In particular, one has:
\begin{equation}\label{p_due}
\mu(\bar\Omega) = \kappa(\Omega)^2, \qquad \nu(\bar\Omega) = 1.
\end{equation}
We now assume that:
\begin{equation}\label{p_tre}
\kappa(\Omega) > \kappa(\mathbb{R}^n),
\end{equation}
and derive a contradiction.
We will distinguish two cases: when $\mu(\Omega) > 0$ and $\mu(\Omega) = 0$.

\medskip

{\bf 2.} First, notice that:
\begin{equation}\label{p_quattro}
\forall \phi\in\mathcal{C}_c^\infty (\mathbb{R}^n) \qquad
\int_{\mathbb{R}^n} \phi^2 ~\mbox{d}\mu \leq \kappa(\Omega)^2 
\int_{\mathbb{R}^n} \phi^2 ~\mbox{d}\nu.
\end{equation}
Indeed, for a given $\phi$ as above consider the sequence 
$v_k = \phi u_k\in W^{1,2}(\Omega,\mathbb{R}^n)$. 
Clearly $v_k\cdot \vec n =0$ on $\partial\Omega$ and by 
Lemma \ref{projection} we have:
\begin{equation}\label{p_cinque}
\left\|\nabla v_k - \mathbb{P}_{L_\Omega}\fint_\Omega 
\nabla v_k\right\|_{L^2(\Omega)} 
\leq \kappa(\Omega) \|D(v_k)\|_{L^2(\Omega)}.
\end{equation}
Since $\nabla v_k = \phi \nabla u_k + u_k\otimes\nabla\phi$, 
the sequence $\fint_\Omega\nabla v_k$ converges to $0$ in $\mathbb{R}^{n\times n}$
by (i). The same convergence must 
be true for the respective sequence of projections.
Similarly, $\lim_{k\to\infty}\|u_k\otimes\nabla \phi\|_{L^2(\Omega)} =0$
by (i).
Hence (\ref{p_cinque}), after passing to the limit with $k\to\infty$
yields:
$$\lim_{k\to\infty}\|\phi\nabla u_k\|_{L^2(\Omega)} 
\leq \kappa(\Omega) \lim_{k\to\infty}\|\phi D(u_k)\|_{L^2(\Omega)},$$
which in view of (\ref{p_uno}) proves (\ref{p_quattro}).

Assume now that $\mu(\Omega)>0$. In this case we are ready 
to derive a contradiction.
Let $B$ be an open ball, compactly contained in $\Omega$, with $\mu(B)>0$.
By (\ref{p_quattro}):
\begin{equation}\label{p_sei}
\mu(\bar\Omega\setminus B) \leq \kappa(\Omega)^2 \nu(\bar\Omega\setminus B).
\end{equation}
On the other hand, recalling the definition (\ref{def_n}) and
reasoning exactly as in the proof of (\ref{p_quattro}), we get:
$$\forall \phi\in\mathcal{C}_c^\infty (B) \qquad \int_B \phi^2~\mbox{d}\mu
\leq \kappa(\mathbb{R}^n)^2 \int_B \phi^2~\mbox{d}\nu,$$
which implies:
\begin{equation}\label{p_sette}
\mu(B)\leq \kappa(\mathbb{R}^n)^2\nu(B).
\end{equation}
Now, both sides of (\ref{p_sette}) are positive, so by (\ref{p_tre}):
$\mu(B) < \kappa(\Omega)^2\nu(B)$. Together with (\ref{p_sei}) this yields:
$$\mu(\bar\Omega) < \kappa(\Omega)^2 \nu(\bar\Omega),$$
contradicting (\ref{p_due}).

\medskip

{\bf 3.} It remains to consider the case $\mu(\Omega) = 0,$ when the measure
$\mu$ concentrates on $\partial\Omega$, due to the lack of the 
equiintegrability of the sequence $\{|\nabla u_k|^2\}_{k=1}^\infty$ 
close to $\partial\Omega$.
We will prove that:
\begin{equation}\label{p_nove}
\mu(\partial\Omega) \leq \kappa(\mathbb{R}^n)^2\nu(\partial\Omega).
\end{equation}
Both sides of (\ref{p_nove}) are positive, and so (\ref{p_tre}) in view of 
the assumption $\mu(\Omega) = 0$ implies:
$$\mu(\bar\Omega) = \mu(\partial\Omega) < \kappa(\Omega)^2\nu(\partial\Omega)
\leq \kappa(\Omega)^2\nu(\bar\Omega), $$
contradicting (\ref{p_due}). This will end the proof of the theorem.

\medskip

Towards establishing (\ref{p_nove}), let $\theta:[0,\infty)\longrightarrow [0,1]$
be a smooth, non-increasing function such that:
$$\theta(t) = 1 \mbox{ for } t\in [0,1], \qquad 
\theta(t) = 0 \mbox{ for } t\geq 2.$$
Define: $\phi_k(x) =\theta(k\mbox{dist}(x,\partial\Omega))$. 
For large $k$ we have $\phi_k\in\mathcal{C}_c^\infty
((\partial\Omega)_\epsilon)$ on a small open neighborhood $(\partial\Omega)_\epsilon$ 
of $\partial\Omega$.
By (\ref{p_uno}), for some increasing sequence $\{n_k\}_{k=1}^\infty$:
\begin{equation}\label{p_otto}
\mu(\partial\Omega) = \lim_{k\to\infty} \|\phi_k\nabla u_{n_k}\|^2_{L^2(\Omega)},
\qquad \nu(\partial\Omega) 
= \lim_{k\to\infty} \|\phi_k D(u_{n_k})\|^2_{L^2(\Omega)}.
\end{equation}
To simplify the notation, we will pass to subsequences and write $n_k = k$.

Define the extension of $u_k$ on $(\partial\Omega)_\epsilon$ 
by reflecting the normal components oddly and 
tangential components evenly, across $\partial\Omega$. 
That is, denoting by $\pi:(\partial\Omega)_\epsilon
\longrightarrow\partial\Omega$ the projection onto $\partial\Omega$ 
along the normal vectors $\vec n$, so that:
$$\big(x-\pi(x) \big) \parallel \vec n(\pi(x)) \qquad \forall x\in (\partial\Omega)_\epsilon,$$
let, for all $x\in (\partial\Omega)_\epsilon\setminus \Omega$:
\begin{equation}\label{p_dieci}
\begin{split}
& u_k(x)\cdot \vec n(\pi(x)) = - u_k(2\pi(x) - x)\cdot\vec n(\pi(x)),\\
& u_k(x)\cdot\tau = u_k(2\pi(x) - x)\cdot\tau\qquad 
\forall \tau\in T_{\pi(x)}\partial\Omega.
\end{split}
\end{equation}
Since $u_k\cdot \vec n = 0$ on $\partial\Omega$, the above defined 
extension $u_k$ is still $W^{1,2}$ regular. By (\ref{def_n}) there holds:
$$\|\nabla (\phi_k u_k)\|_{L^2(\mathbb{R}^n)}\leq \kappa(\mathbb{R}^n)
\|D(\phi_k u_k)\|_{L^2(\mathbb{R}^n)}.$$
Again, by taking $\{n_k\}$ in (\ref{p_otto}) 
converging to $\infty$ sufficiently fast,
we may without loss of generality assume that
$\|u_k\|_{L^2(\Omega)}\leq 1/k^2$. Therefore:
\begin{equation}\label{p_undici}
\lim_{k\to\infty} \|\phi_k\nabla u_k\|_{L^2(\mathbb{R}^n)}
\leq \kappa(\mathbb{R}^n) 
\lim_{k\to\infty}\|\phi_k D(u_k)\|_{L^2(\mathbb{R}^n)}.
\end{equation}
Consider the quantity:
$$I = \lim_{k\to\infty}\Bigg\{\int_{\mathbb{R}^n\setminus\Omega}
|\phi_k\nabla u_k|^2 - \int_\Omega |\phi_k\nabla u_k|^2 \Bigg\}.$$
After changing the variables in the first integral and noting that:
$$\det \nabla (2\pi(x) - x) = \det ~ (2\nabla\pi(x) - \mbox{Id})
= -1 + \mathcal{O}(1) |x-\pi(x)|,$$
we obtain:
\begin{equation}\label{p_dodici}
\begin{split}
I & = \lim_{k\to\infty} \int_\Omega\Big\{|\phi_k(x)\nabla u_k(2\pi(x) - x)|^2
- |\phi_k(x)\nabla u_k(x)|^2\Big\}~\mbox{d}x\\
& = \lim_{k\to\infty}\int_{\Omega\cap\{\mathrm{dist}(x,\partial\Omega)<1/k\}}
\Big\{|\nabla u_k(2\pi(x) - x)|^2 - |\nabla u_k(x)|^2\Big\}~\mbox{d}x.
\end{split}
\end{equation}
The definition of extension (\ref{p_dieci}) yields now the following identities,
for each $x\in (\partial\Omega)_\epsilon$ and each 
$\tau,\eta\in T_{\pi(x)}\partial\Omega$:
\begin{equation}\label{p_tredici}
\begin{split}
& \partial_{\tau} (u_k\cdot\eta) 
(2\pi(x) - x) = \Big(1+\mathcal{O}(1) |x-\pi(x)|\Big)
\partial_{\tau} (u_k\cdot\eta)(x),\\
& \partial_{\vec n(\pi(x))}(u_k\cdot\eta) (2\pi(x) - x) = 
- \partial_{\vec n(\pi(x))}(u_k\cdot\eta) (x),\\
& \partial_{\tau}(u_k \cdot \vec n(\pi(x))(2\pi(x) - x) =
\Big(-1+\mathcal{O}(1) |x-\pi(x)|\Big)
\partial_{\tau}(u_k \cdot \vec n(\pi(x)) (x),\\
& \partial_{\vec n(\pi(x))}(u_k\cdot \vec n(\pi(x))(2\pi(x) - x) =
\partial_{\vec n(\pi(x))}(u_k\cdot \vec n(\pi(x)) (x).
\end{split}
\end{equation}
Since $\eta\partial_\tau v_k = \partial_\tau (v_k\eta) 
- v_k\partial_\tau \eta$, we see that equating the contribution 
of all components in (\ref{p_dodici}) and recalling (iii) we have:
$$I=0.$$
In the same manner, (\ref{p_tredici}) implies that $|D(u_k)(2\pi(x)-x)|^2$
equals to $|D(u_k)(x)|^2$ plus lower order terms whose integrals on
$\Omega\cap\{\mbox{dist}(x,\partial\Omega)<1/k\}$ vanish, as 
$k\longrightarrow \infty$.
Hence also:
$$II = \lim_{k\to\infty}\Bigg\{\int_{\mathbb{R}^n\setminus\Omega}
|\phi_k D(u_k)|^2 - \int_\Omega |\phi_k D(u_k)|^2 \Bigg\} = 0.$$
Therefore:
\begin{equation}\label{p_quattordici}
\begin{split}
& \lim_{k\to\infty} \|\phi_k \nabla u_k\|_{L^2(\mathbb{R}^n)}
= 2 \lim_{k\to\infty}\|\phi_k \nabla u_k\|_{L^2(\Omega)},\\
& \lim_{k\to\infty} \|\phi_k D(u_k)\|_{L^2(\mathbb{R}^n)}
= 2 \lim_{k\to\infty}\|\phi_k D(u_k)\|_{L^2(\Omega)}.
\end{split}
\end{equation}
Combining (\ref{p_quattordici}), (\ref{p_undici}) with (\ref{p_otto}) 
proves (\ref{p_nove}).
\endproof

\medskip

\noindent{\bf Proof of Theorem \ref{cor1}}

\noindent It is enough to assume that $A_0=0$. 
Let $\{u_k\}_{k=1}^\infty$ be a maximizing sequence of (\ref{def_kappa}),
that is: $u_k\in W^{1,2}(\Omega, \mathbb{R}^n)$, $u_k\cdot\vec n = 0$ 
on $\partial\Omega$, $\|D(u_k)\|_{L^2(\Omega)}=1$ and
$\lim_{k\to\infty} \left\|\nabla u_k  - \mathbb{P}_{L_\Omega}
\fint_\Omega\nabla u_k\right\|_{L^2(\Omega)} = \kappa(\Omega)$.

By modifying $u_k$ we may, without loss of generality, assume that:
\begin{equation}\label{cor1_0}
\mathbb{P}_{L_\Omega}\fint\nabla u_k= 0,
\qquad \lim_{k\to\infty} \|\nabla u_k \|_{L^2(\Omega)} = \kappa(\Omega).
\end{equation}
Using Lemma \ref{poincare} (after possibly passing to a subsequence), 
we have:
\begin{equation}\label{cor1_1}
u_k \rightharpoonup u \quad \mbox{ weakly in } W^{1,2}(\Omega, \mathbb{R}^n),
\end{equation}
for some $u$ satisfying $u\cdot\vec n = 0$ on $\partial\Omega$.

We now show that (\ref{exact}) holds with $A_0=0$.
First of all, by applying Theorem \ref{th_due} to the sequence $\{u_k\}$, 
we see that $u\neq 0$. 
Further, (\ref{cor1_1}) implies that 
$\mathbb{P}_{L_\Omega} \fint \nabla u = 
\lim_{k\to\infty} \mathbb{P}_{L_\Omega} \fint \nabla u_k = 0$, so:
\begin{equation}\label{cor1_2}
\|\nabla u\|_{L^2(\Omega)} \leq \kappa(\Omega) \|D(u)\|_{L^2(\Omega)}.
\end{equation}
Since $\mathbb{P}_{L_\Omega}\fint\nabla (u_k-u)= 0$, there also holds:
$$\|\nabla (u_k-u)\|_{L^2(\Omega)} \leq \kappa(\Omega) \|D(u_k-u)\|_{L^2(\Omega)}.$$
Squaring both sides of the above inequality, passing to the limit
with $k\to \infty$ and recalling (\ref{cor1_0}) and (\ref{cor1_1}),
we obtain:
$$\kappa(\Omega)^2 - \|\nabla u\|^2_{L^2(\Omega)}\leq
\kappa(\Omega)^2\left(1 - \|D(u)\|^2_{L^2(\Omega)}\right).$$
Together with (\ref{cor1_2}) this proves:
$$\|\nabla u\|_{L^2(\Omega)} = \kappa(\Omega)\|D(u)\|_{L^2(\Omega)},$$
yielding the result.
\endproof

\medskip

\noindent{\bf Proof of Theorem \ref{cor2}}

\noindent
{\bf 1.} Let $E$ be the set in (\ref{space_exact}). It is clear that $u\in E$ implies
$\lambda u\in E$, for all $\lambda\in\mathbb{R}$. If $u_1, u_2\in E$, then
by Lemma \ref{projection}:
\begin{equation}\label{c_1}
\left\|\nabla u_i - \mathbb{P}_{L_\Omega} \fint_\Omega \nabla u_i \right\|_{L^2(\Omega)}
= \kappa(\Omega) \|D(u_i)\|_{L^2(\Omega)} \qquad \forall i=1,2.
\end{equation}
On the other hand, by the linearity of the operator $\mathbb{P}_{L_\Omega}$ and
by (\ref{korn}), (\ref{def_kappa}):
\begin{equation}\label{c_2}
\left\|\nabla (u_1\pm u_2) - \left(\mathbb{P}_{L_\Omega} \fint_\Omega \nabla u_1 \pm
\mathbb{P}_{L_\Omega} \fint_\Omega \nabla u_2\right) \right\|_{L^2(\Omega)}
\leq \kappa(\Omega) \|D(u_1\pm u_2))\|_{L^2(\Omega)}.
\end{equation}
Squaring the two inequalities in (\ref{c_2}) and equating the terms from (\ref{c_1})
we obtain:
$$\left\langle \nabla u_1 -  \mathbb{P}_{L_\Omega} \fint_\Omega \nabla u_1,
~ \nabla u_2 -  \mathbb{P}_{L_\Omega} \fint_\Omega \nabla u_2\right\rangle_{L^2(\Omega)}
= \kappa(\Omega)^2 \left\langle D(u_1), D(u_2)\right\rangle_{L^2(\Omega)}.$$
Therefore, (\ref{c_2}) is actually true as the equality.
We hence conclude that $u_1+ u_2\in E$, proving that $E$ is a linear
space.

The closedness of $E$ follows by noting that if a sequence $u_k$ converges to $u$ in 
$W^{1,2}(\Omega,\mathbb{R}^n)$ then the minimizing matrices 
$\mathbb{P}_{L_\Omega} \fint \nabla u_k$ converge to 
$\mathbb{P}_{L_\Omega} \fint \nabla u$.

\medskip

{\bf 2.} To prove the second claim, we argue by contradiction.
Assume that the space $E$ is of infinite dimension. Then it admits a Hilbertian
(orthonormal in $W^{1,2}(\Omega,\mathbb{R}^n)$) basis $\{u_k\}_{k=1}^\infty$.
It is easy to see that there must be:
\begin{equation}\label{c_4}
u_k \rightharpoonup 0 \quad \mbox{ weakly in } 
W^{1,2}(\Omega, \mathbb{R}^n).
\end{equation}
We now notice that:
\begin{equation}\label{c_5}
\liminf_{k\to\infty} \|D(u_k)\|_{L^2(\Omega)} > 0.
\end{equation}
Because otherwise, by Korn's inequality (\ref{korn}) there would be:
$$\liminf_{k\to\infty} \left\|\nabla u_k 
- \mathbb{P}_{L_\Omega}\fint_\Omega\nabla u_k\right\|_{L^2(\Omega)} = 0,$$
and since by (\ref{c_4}) $\lim_{k\to\infty}\fint\nabla u_k = 0$, there
follows that $\liminf_{k\to\infty}\|\nabla u_k \|_{L^2(\Omega)} = 0$.
In view of the Poincar\'e inequality (\ref{poinc}), we hence obtain
$\liminf_{k\to\infty}\|u_k\|_{W^{1,2}(\Omega)}=0$,
in contradiction with the orthonormality of the sequence $\{u_k\}_{k=1}^\infty$.

\medskip

Define: $v_k = u_k/\|D(u_k)\|_{L^2(\Omega)}$. Clearly, there holds:
$$\|D(v_k)\|_{L^2(\Omega)} =1, \qquad \|\nabla v_k\|_{L^2(\Omega)} 
= \kappa(\Omega),$$
and because of (\ref{c_5}) we also have: 
$ v_k \rightharpoonup 0$ weakly in $W^{1,2}(\Omega,\mathbb{R}^n)$.
By Theorem \ref{th_due} there follows $\kappa(\Omega) =
\kappa(\mathbb{R}^n) = \sqrt{2}$,
which is a desired contradiction.
\endproof

\section{The optimal geometric  rigidity constant in $\mathbb{R}^2$}

To prove Theorems \ref{maxmin}, \ref{th2}, \ref{th3}
we need some preliminary discussion.

\begin{lemma}\label{after}
Assume that $w\in L^2_{loc}(\mathbb{R}^n)$ and $\Delta w = 0$. 
If $w=f+g$ with $f\in L^2(\mathbb{R}^n)$ and $g\in L^\infty
(\mathbb{R}^n)$, then $w\equiv const$.
\end{lemma}
\begin{proof}
Fix $x_0, y_0\in\mathbb{R}^n$. For any $r>0$ we have:
\begin{equation*}
\begin{split}
|w(x_0) - w(y_0)| & = \left |\fint_{B_r(x_0)} w - \fint_{B_r(y_0)} w\right| 
= \frac{1}{|B_r|} \left|\int_{B_r(x_0) \Delta B_r(y_0)} w \right| \\ &
\leq \left(\frac{1}{|B_r|} \int_{B_r(x_0) \Delta B_r(y_0)} | f | \right) 
+ \left(\frac{1}{|B_r|} \int_{B_r(x_0) \Delta B_r(y_0)} | g | \right)
\\ & \leq \frac{|B_r(x_0)\Delta B_r(y_0)|^{1/2}}{|B_r|} \|f\|_{L^2}
+  \frac{|B_r(x_0)\Delta B_r(y_0)|}{|B_r|} \|g\|_{L^\infty} \\ &
\leq \left(\frac{1}{|B_r|}  + \frac{|B_r(x_0)\Delta B_r(y_0)|}{|B_r|}
\right) \left(\|f\|_{L^2} + \|g\|_{L^\infty}\right),
\end{split}
\end{equation*}
where by $\Delta$ we denote the symmetric difference of two sets:
$B_1\Delta B_2 = (B_1\setminus B_2)\cup (B_2\setminus B_1)$.
The quantity in the first parentheses above clearly converges to $0$ as
$r\to\infty$.
Therefore $w(x_0) = w(y_0)$, which achieves the proof.
\end{proof}

\begin{lemma}\label{dec_approx}
Let $f\in L^2(\mathbb{R}^2, \mathbb{R}^{2\times 2})$ and let $u\in
W^{1,2}_{loc}(\mathbb{R}^2,\mathbb{R}^2)$ satisfy:
\begin{equation}\label{jeden}
-\Delta u = \mathrm{div } ~f \qquad \mbox{in }  \mathcal{D}'(\mathbb{R}^2).
\end{equation}
Then $u$ can be decoupled as:
\begin{equation}\label{dwa}
u = v + w; \qquad v, w \in L^2_{loc}, ~~ \nabla v \in L^2, ~~ \nabla w
\in L^2_{loc}, ~~ -\Delta w = 0 ~\mbox{ in } \mathbb{R}^2.
\end{equation}
Moreover:
\begin{equation}\label{trzy}
\nabla v = \lim_{m\to\infty}\nabla v_m
\qquad \mbox{ strongly in } L^2(\mathbb{R}^n),
\mbox{ for some }  ~ v_m \in \mathcal{C}_c^\infty(\mathbb{R}^2,\mathbb{R}^2).
\end{equation}
\end{lemma}
\begin{proof}
For each $m\in\mathbb{N}$, let $v_m$ be the solution to: 
\begin{equation}\label{cztery}
\left\{\begin{array}{l} v_m \in W_0^{1,2}(B_m)\\ \displaystyle{
-\int_{\mathbb{R}^2} \nabla v_m : \nabla \phi = \int_{\mathbb{R}^2} f
: \nabla \phi \qquad \forall \phi\in W_0^{1,2}(B_m),}
\end{array}\right.
\end{equation}
whose existence and uniqueness follow from the Lax-Milgram theorem,
together with:
$$\|\nabla v_m\|_{L^2} \leq \|f\|_{L^2(B_m)} \leq \|f\|_{L^2(\mathbb{R}^2)}.$$
Therefore, passing to a subsequence:
\begin{equation}\label{piec}
\nabla v_m \rightharpoonup z \quad \mbox{ weakly in } L^2(\mathbb{R}^2)
\end{equation}
and also :
\begin{equation}\label{szesc}
\mbox{curl } z = 0 ~~ \mbox{ in } \mathcal{D}'(\mathbb{R}^2).
\end{equation}

\smallskip

Condition (\ref{szesc}) is now equivalent to: $z = \nabla
v$. This can be seen, for example, via Helmholtz decomposition
\cite{farwig}:
$$z = z_0 + \nabla v; \qquad v\in L^2_{loc}, ~~ \nabla v, z_0\in L^2,
~~ \mbox{div } z_0 = 0 \mbox{ in } \mathcal{D}'.$$
Since from (\ref{szesc}) also $\mbox{curl } z_0=0$, hence the
components of $z_0$ satisfy the Cauchy-Riemann equations, and therefore $\Delta
z_0 = 0$. Recalling that $z_0\in L^2(\mathbb{R}^2)$ it follows by
Lemma \ref{after} that $z_0 = 0$.
Consequently, by (\ref{piec}):
\begin{equation}\label{siedem}
\nabla v_m \rightharpoonup \nabla v \quad \mbox{ weakly in } L^2(\mathbb{R}^2).
\end{equation}
Passing to the limit in (\ref{cztery}), we obtain: $-\Delta v =
\mbox{div } f$ in $\mathcal{D}'$, hence $-\Delta w = 0$, for $w=u-v$
and (\ref{dwa}) is proven.

\smallskip

Finally, by Mazur's theorem and (\ref{siedem}), $\nabla v$
is the strong $L^2$-limit of $\nabla \tilde v_m$ which are gradients of
some finite (in fact, convex) linear combinations $\tilde v_m$ of
$v_m$. Clearly, each $\tilde v_m\in W^{1,2}_0(B_{r_m})$ and the result in
(\ref{trzy}) follows by density of $\mathcal{C}_c^\infty(B_{r_m})$ in $W_0^{1,2}(B_{r_m})$. 
\end{proof}

\begin{remark}
Note that one can directly show that $\nabla v_m$ in Lemma
\ref{dec_approx} converges strongly in $L^2(\mathbb{R}^2)$.
Let $k > m$. Extending $v_m$ by zero to $\mathbb{R}^2$, so that $v_m \in W^{1,2}_0(B_k)$, and taking
$\phi = v_m$ in the equation (\ref{cztery}) written for $v_k$, we get:
$$ \int_{B_k} \nabla v_k : \nabla v_m = - \int_{B_k} f : \nabla v_m = - \int_{B_m}
f : \nabla v_m = \int_{B_m} |\nabla v_m|^2. $$
The last equality above follows by taking $\phi = v_m$ in the equation
(\ref{cztery}) written for $v_m$.
Now, passing $m \to \infty$ implies, by the weak convergence in (\ref{piec}):
$$ \lim_{m \to \infty}  \int_{\mathbb{R}^2} |\nabla v_m|^2 =
\int_{\mathbb{R}^2} \nabla v_k : \nabla v.$$
Finally, passing $k \to \infty$ yields:
$$ \lim_{m \to \infty}  \int_{\mathbb{R}^2} |\nabla v_m|^2 = \int_{\mathbb{R}^2} | \nabla v|^2.  $$
The claim (\ref{trzy}) now follows, since convergence of norms
in presence of the weak convergence implies 
strong convergence in $L^2(\mathbb{R}^2)$.  
\end{remark}

\begin{lemma}\label{lemdet}
Let $u\in W_{loc}^{1,2} (\mathbb{R}^2, \mathbb{R}^2)$ and $\nabla u\in
L^2(\mathbb{R}^2)$. Then:
\begin{equation}\label{det}
\int_{\mathbb{R}^2} \det \nabla u = 0.
\end{equation}
\end{lemma}
\begin{proof}
Since $\Delta u = \mbox{div } \nabla u$ in
$\mathcal{D}'(\mathbb{R}^2)$ we may apply Lemma \ref{dec_approx} to
$f=\nabla u \in L^2(\mathbb{R}^2)$ and write $u = v+ w$ satisfying (\ref{dwa}). 
Since $\Delta w = 0$ and $\nabla w = \nabla u - \nabla v\in L^2(\mathbb{R}^2)$, it
follows from Lemma \ref{after} that $\nabla w = 0$ and hence by
(\ref{trzy}):
$$\nabla u = \nabla v = \lim_{m\to\infty} \nabla v_m \quad \mbox{ strongly in
  $L^2(\mathbb{R}^2)$, for some }
v_m\in\mathcal{C}_c^\infty(\mathbb{R}^2, \mathbb{R}^2).$$
It remains to prove (\ref{det}) for $u\in
\mathcal{C}_c^\infty$, which is a standard argument.
Let $\mbox{supp } u \subset B_r$. We have:
\begin{equation*}
\begin{split}
\int_{\mathbb{R}^2} \det \nabla u & = \int_{B_r} (\partial_1
u^1\partial_2 u^2 - \partial_1 u^2 \partial_2 u^1) \\ & = 
 \int_{B_r} (\partial_1(
u^1\partial_2 u^2) - \partial_2(u^1 \partial_1 u^2)) =
\int_{\partial B_r} (u^1\partial_2 u^2,  u^1 \partial_1 u^2)\vec n =0, 
\end{split}
\end{equation*}
where we used integration by parts and the divergence theorem.
\end{proof}

\medskip

We finally need to recall the conformal--anticonformal decomposition of $2\times 2$
matrices.
Let $\mathbb{R}^{2\times 2}_c$ and $\mathbb{R}^{2\times 2}_a$ denote,
respectively,  the spaces of conformal and anticonformal matrices:
$$\mathbb{R}^{2\times 2}_c =
\left\{\left[\begin{array}{cc} a& b\\ -b& a\end{array}\right]; ~ a,b\in\mathbb{R}\right\},
\qquad \mathbb{R}^{2\times 2}_a 
= \left\{\left[\begin{array}{cc} a& b\\ b& -a\end{array}\right]; ~ a,b\in\mathbb{R}\right\}.$$
It is easy to see that $\mathbb{R}^{2\times 2} =
\mathbb{R}^{2\times 2}_c \oplus \mathbb{R}^{2\times 2}_a$ because both
spaces have dimension $2$ and they are mutually orthogonal: $A:B = 0$ for all
$A\in  \mathbb{R}^{2\times 2}_c$ and $B\in \mathbb{R}^{2\times 2}_a$. 

For $F=[F_{ij}]_{i,j:1,2}\in\mathbb{R}^{2\times 2}$, its projections
$F^c $ on $\mathbb{R}^{2\times 2}_c$, and $F^a$ on $\mathbb{R}^{2\times 2}_a$ are:
$$F^c= \frac{1}{2}
\left[\begin{array}{cc} F_{11} + F_{22} & F_{12} - F_{21}\\ F_{21} -
    F_{12} & F_{11} + F_{22} \end{array}\right],
\qquad 
F^a= \frac{1}{2}
\left[\begin{array}{cc} F_{11} - F_{22} & F_{12} + F_{21}\\ F_{12} +
    F_{21} & F_{22} - F_{11} \end{array}\right].$$
It follows that:
\begin{equation}\label{uno1} 
F= F^c + F^a \quad \mbox{ and } \quad |F|^2 = |F^c|^2 + |F^a|^2
\end{equation}
and, by a direct calculation:
\begin{equation}\label{due1} 
\det F= 2( |F^c|^2 - |F^a|^2).
\end{equation}
Since $SO(2) = \left\{\left[\begin{array}{cc} \cos\theta & -
      \sin\theta\\ \sin \theta& \cos\theta  \end{array}\right], ~\theta\in
    [0,2\pi)\right\}\subset \mathbb{R}^{2\times 2}_c$, it also follows that:
\begin{equation}\label{duea} 
\mbox{dist}(F, SO(2)) \geq \mbox{dist}(F,\mathbb{R}^{2\times 2}_c) =
|F^a|\end{equation}
which implies:
\begin{equation}\label{tre1} 
|\mbox{cof } F - F| = |-2 F^a| \leq 2 \mbox{dist}(F, SO(2)).
\end{equation}
Finally, recall that the cofactor matrix  in dimension $2$ is given by: 
$$\mbox{cof } F = \left[\begin{array}{cc} F_{22} & -
    F_{21}\\ -F_{12}& F_{11} \end{array}\right].$$

\medskip

We now state the following first result towards proving Theorem \ref{maxmin}.
\begin{lemma}\label{th1}
Let $u\in W^{1,2}_{loc}(\mathbb{R}^2, \mathbb{R}^2)$ and assume that 
$\mathrm{dist}(\nabla u, SO(2))\in L^2(\mathbb{R}^2,\mathbb{R})$. Then
there exists $R_0\in SO(2)$ such that:
$$ \int_{\mathbb{R}^2} |\nabla u(x) - R_0|^2~\mathrm{d}x \leq
2 \int_{\mathbb{R}^2}\mathrm{dist}^2(\nabla u(x), SO(2))~\mathrm{d}x.$$
\end{lemma}
\begin{proof}
From the assumption $\mathrm{dist}(\nabla u, SO(2))\in L^2(\mathbb{R}^2)$
and (\ref{tre1}) we deduce: 
$$ f := \mbox{cof } \nabla u - \nabla u \in L^2(\mathbb{R}^2).$$
Taking divergence of $f$ and recalling that $\mbox{div } \mbox{cof }
\nabla u = 0$ we obtain that $-\Delta u = \mbox{div }f$. In view of
Lemma \ref{dec_approx} we now write:
\begin{equation}\label{quattro}
u = v + w
\end{equation}
where $v$ and $w$ satisfy (\ref{dwa}).
We now prove that:
\begin{equation}\label{otto}
\nabla w \equiv R_0 \in SO(2).
\end{equation}
For $\epsilon>0$ sufficiently small, define:
\begin{equation*}
g(x) = \left\{\begin{array}{ll} \mathbb{P}_{SO(2)} \nabla u(x) &
    \mbox{ if } \mbox{dist}(\nabla u(x), SO(2))<\epsilon\\
\mbox{Id} & \mbox{otherwise}\end{array} \right.
\end{equation*}
Then:
\begin{equation}\label{nove}
\nabla w = g + h; \quad g\in L^\infty(\mathbb{R}^2) \mbox{ and } h\in L^2(\mathbb{R}^2).
\end{equation}
The assertion $h= \nabla w - g = \nabla u - g + \nabla v\in L^2(\mathbb{R}^2)$
follows from the assumption $\mathrm{dist}(\nabla u, SO(2))\in L^2(\mathbb{R}^2)$ as
follows.
We already know that $\nabla v \in L^2(\mathbb{R}^2)$ by (\ref{quattro}). For
$h_1=\nabla u - g$ note that $|h_1(x)| = \mathrm{dist}(\nabla u,
SO(2))$ when $\mathrm{dist}(\nabla u(x), SO(2)) < \epsilon$, while
when $\mathrm{dist}(\nabla u(x), SO(2)) \geq \epsilon$, we have:
\begin{equation*}
\begin{split}
|h_1(x)| & = |\nabla u(x) - \mbox{Id}| \leq \mathrm{dist}(\nabla u(x),
SO(2)) + \mbox{diam} \big(SO(2)\big) \\ & \leq \left( 1 + 
\frac{\mbox{diam}\big( SO(2)\big)}{\epsilon}\right) 
\mathrm{dist}(\nabla u(x), SO(2)).
\end{split}
\end{equation*}
Since $\nabla w$ is harmonic in $\mathbb{R}^2$, (\ref{nove}) implies
that $\nabla w\equiv R_0$ is constant by Lemma \ref{after}.
But $\mathrm{dist}(R_0, SO(2))\leq \mathrm{dist}(\nabla u, SO(2)) +
|\nabla v|\in L^2(\mathbb{R}^2)$, so $ R_0\in SO(2)$ and (\ref{otto}) is now established.

\smallskip

By (\ref{otto}) and (\ref{quattro}) we have:
\begin{equation}\label{dieci}
\nabla u = \nabla v + R_0.
\end{equation}
Since $\int \det\nabla v =0$ by Lemma \ref{lemdet}, we obtain by
(\ref{due1}):
\begin{equation}\label{undici}
\int_{\mathbb{R}^2} |(\nabla v)^c|^2 = \int_{\mathbb{R}^2} |(\nabla v)^a|^2.
\end{equation}
Consequently:
\begin{equation*}
\begin{split}
\int_{\mathbb{R}^2} |\nabla u - R_0|^2 & = \int |\nabla
v)|^2 = \int |(\nabla v)^c|^2 + \int |(\nabla v)^a|^2 = 2 \int |(\nabla v)^a|^2 \\ &
= 2 \int |(\nabla v + R_0)^a|^2 \leq 2
\int \mbox{dist}^2 (\nabla v + R_0, SO(2)) \\ & 
= 2 \int_{\mathbb{R}^2} \mbox{dist}^2 (\nabla u, SO(2)), 
\end{split}
\end{equation*}
where we used (\ref{dieci}), (\ref{uno1}), (\ref{undici}), (\ref{duea})
and the fact that $(R_0)^a = 0$. This achieves the proof.
\end{proof}

\medskip

\noindent{\bf Proof of Theorem \ref{th2}}

\noindent {\bf 1.} Without loss of generality we may assume that $R_0 =
\mbox{Id}$. We shall look for a function $u\in
W^{1,2}_{loc}(\mathbb{R}^2, \mathbb{R}^2)$ such that:
\begin{equation}\label{dodici2}
\nabla u(x) = R(\alpha(x)) + \left[\begin{array}{cc} a(x) & b(x)\\
b(x) & -a(x)\end{array}\right], \quad \mbox{with } R(\alpha) = 
\left[\begin{array}{cc} \cos \alpha & -\sin\alpha\\
\sin\alpha & \cos\alpha\end{array}\right],
\end{equation}
and:
\begin{equation}\label{tredici}
\nabla u(x) -\mbox{Id} \in L^2(\mathbb{R}^2)
\end{equation}
Indeed, note that by Lemma \ref{lemdet}, (\ref{tredici}) and
(\ref{due1}):
$$\int_{\mathbb{R}^2} |(\nabla u)^c - \mbox{Id}|^2 = \int_{\mathbb{R}^2} |(\nabla u)^a|^2.$$
Hence, by (\ref{uno1}):
$$\int_{\mathbb{R}^2} |\nabla u - \mbox{Id}|^2 = 2\int_{\mathbb{R}^2}
|(\nabla u)^a|^2 = 2\int_{\mathbb{R}^2} \mathrm{dist}^2(\nabla u, SO(2)).$$
because $(\nabla u)^c = R(\alpha) \in SO(2)$. Since $(\nabla
u)^a = \left[\begin{array}{cc} a & b\\ b & -a\end{array}\right]$ it also
follows that $\int_{\mathbb{R}^2} \mathrm{dist}^2(\nabla u, SO(2)) = 2 
\int_{\mathbb{R}^2} (a^2 + b^2)$.

On the other hand, there is always the unique rotation $R$ which makes
the quantity in the left hand side of (\ref{attain}) finite:
$$\int_{\mathbb{R}^2} |\nabla u - R|^2 \geq \frac{1}{2}
\int_{\mathbb{R}^2} | R - \mbox{Id}|^2  - \int_{\mathbb{R}^2} |\nabla u - \mbox{Id}|^2. $$
This proves the theorem, provided (\ref{dodici2}) and (\ref{tredici})
hold.

\smallskip 

{\bf 2.} We shall show that for any $\alpha\in L^2(\mathbb{R}^2,
\mathbb{R})$ there exists a vector field $g = (a, b)^T \in L^2(\mathbb{R}^2,
\mathbb{R}^2)$ satisfying (\ref{dodici2}). Then (\ref{tredici}) will
follow automatically, as:
$$\int |R(\alpha) - \mbox{Id}|^2 = 2\int (\cos\alpha -1)^2 +
(\sin\alpha)^2 = 2\int (2-2\cos\alpha) \leq 2 \int |\alpha|^2.$$
The last inequality above follows by noting that the function
$\alpha\mapsto \alpha^2 + 2\cos\alpha - 2$ attains its  minimum value
$0$ at $\alpha=0$, since $(\alpha^2 + 2\cos\alpha - 2)' = 2(\alpha
-\sin\alpha)$
is positive for $\alpha>0$ and negative for $\alpha<0$.

Fix $\alpha\in L^2(\mathbb{R}^2)$. The map  $u\in
W^{1,2}_{loc}(\mathbb{R}^2, \mathbb{R}^2)$ with $\nabla u$ of
the form (\ref{dodici}) exists if and only if the right hand side in
(\ref{dodici}) is curl-free, i.e.:
\begin{equation}\label{quattordici}
\left\{ \begin{array}{ll}  \mbox{curl } g = \mbox{div } f & \qquad
    \mbox{in } \mathcal{D}'(\mathbb{R}^2)\\ \mbox{div } g = \mbox{curl
    } f & \end{array}\right.
\end{equation}
where:
$$f = (\sin\alpha,~ \cos\alpha -1)^T\in L^2(\mathbb{R}^2, \mathbb{R}^2).$$
The system (\ref{quattordici}) can be solved by Fourier transform,
namely:
\begin{equation}\label{quindici}
g=\mathcal{F}^{-1}(h), \qquad h(x) = - \left\langle \frac{x^\perp}{|x|},
\mathcal{F}(f)(x)\right\rangle \frac{x}{|x|} + \left\langle \frac{x}{|x|},
\mathcal{F}(f)(x)\right\rangle \frac{x^\perp}{|x|},
\end{equation}
where $x^\perp = (-x_2, x_1)$. Here $\mathcal{F}$ stands for the
Fourier transform of $L^2(\mathbb{R}^2, \mathbb{C})$ and we identify
the complex variable functions with the $\mathbb{R}^2$-valued vector fields.

Note that from (\ref{quindici}) it follows that: 
\begin{equation}
\begin{array}{ll}
\forall x\in\mathbb{R}^2 & \langle \mathcal{F}(g)(x), x^\perp\rangle = 
\langle \mathcal{F}(f)(x), x\rangle\\
& \langle \mathcal{F}(g)(x), x\rangle = 
- \langle \mathcal{F}(f)(x), x^\perp\rangle
\end{array}
\end{equation}
which precisely implies (\ref{quattordici}).
Therefore, for every $f\in L^2(\mathbb{R}^2)$ there exists a unique $g\in L^2(\mathbb{R}^2)$
solving (\ref{quattordici}). This achieves the proof of the
theorem. Moreover: 
$$ \|g\|_{L^2} = \|h\|_{L^2} = \|\mathcal{F}(f)\|_{L^2}
= \|f\|_{L^2},$$
by Plancherel identity and by inspecting
(\ref{quindici}). 
\endproof

\medskip

This concludes the proof of Theorem \ref{maxmin} as well. In view of
the argument in the above proof, Theorem \ref{th3} will follow in view of:

\begin{lemma}
If $\int |\nabla u - R_0|^2 = 2\int\mathrm{dist}^2(\nabla u, SO(2)) <
\infty$ then $\nabla u$ must be of the form (\ref{dodici2}) with
$R(\alpha) - R_0 \in L^2(\mathbb{R}^2).$
\end{lemma}
\begin{proof}
Note that by (\ref{uno1}): $|\nabla u - R|^2 = |(\nabla u)^c - R|^2 +
|(\nabla u)^a|^2$ for any $R\in SO(2)$. Hence taking infimum over
all rotations, we get:
\begin{equation}\label{sth}
\mbox{dist}^2(\nabla u, SO(2)) = \mbox{dist}^2((\nabla u)^c, SO(2)) +
|(\nabla u)^a|^2.
\end{equation}
In particular: 
\begin{equation*}
(\nabla u)^a\in L^2(\mathbb{R}^2).
\end{equation*}
Further, by (\ref{due1}) and Lemma \ref{lemdet}:
$$\int |(\nabla u)^c - R_0|^2 = \int |(\nabla u)^a|^2.$$
Therefore, by (\ref{uno1}) and (\ref{sth}):
\begin{equation*}
\begin{split}
\int  |(\nabla u)^c - R_0|^2 & = \frac{1}{2} \int  |\nabla u - R_0|^2 = 
\int \mbox{dist}^2(\nabla u, SO(2)) \\ & = \int \mbox{dist}^2((\nabla u)^c, SO(2)) +
\int |(\nabla u)^a|^2 \\ & = \int \mbox{dist}^2((\nabla u)^c, SO(2)) +
\int  |(\nabla u)^c - R_0|^2,
\end{split}
\end{equation*}
which implies that $\int \mbox{dist}^2((\nabla u)^c, SO(2)) =0$ and
hence: $(\nabla u(x))^c\in SO(2)$ for a.e. $x$.
Consequently, $\nabla u$ has the form in (\ref{dodici2}) and:
$$ R(\alpha) - R_0 = \nabla u - R_0 - (\nabla u)^a \in L^2(\mathbb{R}^2)$$
by (\ref{sth}).
\end{proof}

\bigskip

\noindent{\bf Acknowledgments.} 
M.L. was partially supported by the NSF grants DMS-0707275 and DMS-0846996,
and by the Polish MN grant N N201 547438.

\end{document}